\documentclass[11pt, a4paper]{article}
	
\usepackage{amsmath, amsthm, amscd, amsfonts, amssymb, graphicx, color}
\usepackage[bookmarksnumbered, plainpages]{hyperref}

 \newtheorem{thm}{Theorem}

\begin{document}
\centerline{\Large{\bf A Periodic Solution to Impulsive Logistic}}
\centerline{}
\centerline{\Large{\bf Equation}}
\centerline{}
\centerline{Gyong-Chol Kim, Hyong-Chol O, Sang-Mun Kim and Chol Kim}
 \centerline{}
 \small \centerline{Faculty of Mathematics, \textbf{Kim Il Sung} University, Pyongyang, D.P.R Korea}

\centerline{}
\centerline{}
\begin{abstract}
 In this paper is provided a new representation of periodic solution to the impulsive Logistic equation considered in \cite{zha1}.
\end{abstract}

\hspace{1.3cm} Impulse  

\hspace{1.6cm} Secondary 34B37, 34K45
%
%
%
%
\section{Introduction}
There are a lot of works about the solutions and periodic solutions to Logistic equations with initial conditions. 
In \cite{fan} they obtained the solutions and periodic solutions to initial value problem of Logistic equation without impulse action. In \cite{che, zha2} they studied a periodic solution to Logistic equation in the case with a constant breeding 
coefficient ($r$ in Eq. \eqref{eq1} in what follows) and a constant saturation coefficient 
($K$ in Eq. \eqref{eq1} in what follows). \cite{li1}-\cite{yan} provided the similar study.

In \cite{zha1} they studied the following impulsive Logistic equation
\begin{eqnarray}
&& \left\{ \begin{array}{ll}
	x' = r(t) \left[ 1- \frac{x}{K(t)} \right]x, t \neq \tau_k & \qquad \textnormal{(a)}  \\
	\Delta x = -Ex, \quad t=\tau_k (= t_0+k, k=1, 2, \cdots) & \qquad \textnormal{(b)} 
	\end{array} \right. \label{eq1} \\
&& \quad ~ x(t_0) = x_0 \nonumber
\end{eqnarray}
and obtained the representation of solution 
\begin{eqnarray*}
x(t) & = & \left[ \frac{1}{x_0[(1-E)A]^k} \textnormal{exp}\left( -\int_{t_0+k}^t r(\tau)d\tau \right) \right. \\
&& \quad + AB\sum_{j=1}^k \frac{1}{[(1-E)A]^j} \textnormal{exp}\left( \int_{t_0+k}^t r(\tau)d\tau \right) \\
&& \quad \left. + \int_{t_0+k}^t \frac{r(s)}{K(s)} \textnormal{exp}\left( -\int_s^t r(\tau)d\tau \right)ds \right]^{-1},
\end{eqnarray*} 
where $k$ is the number of impulse points in the interval $[t_0, t)$, and the analytical representation of periodic solution 
\begin{equation*}
x_*(t) = [(1-E)A-1] \left[ A \int_t^{t+1} \frac{r(s)}{K(s)} \textnormal{exp}\left( -\int_s^{t+1} r(\tau)d\tau \right)ds\right]^{-1}
\end{equation*}  
with the period 1 when $(1-E)A>1$. Here, $A$ and $B$ are provided as follows:
\begin{eqnarray}
A & = & \textnormal{exp} \int_{t_0+k}^{t_0+k+1} r(\tau)d\tau = \textnormal{exp} \int_0^1 r(\tau)d\tau \nonumber \\
B & = & \int_{t_0+k}^{t_0+k+1} \frac{r(s)}{K(s)} \textnormal{exp} \left( -\int_s^{t_0+k+1} r(\tau)d\tau \right)ds \label{eq2} \\
& = & \int_{t_0}^{t_0+1} \frac{r(s)}{K(s)} \textnormal{exp} \left( -\int_s^{t_0+1} r(\tau)d\tau \right)ds. \nonumber
\end{eqnarray}
But this $x_*(t)$ does not satisfy the impulsive condition (b). In fact
\begin{equation*}
\Delta x_* = x_*\left((t_0+k)^+\right) - x_*\left((t_0+k)^-\right)
\end{equation*}  
and the first term of it is given by
\begin{eqnarray*}
&& x_*\left((t_0+k)^+\right) = \lim_{t \to t_0+k+0} x_*(t) \\
&& \qquad = \lim_{t \to t_0+k+0} [(1-E)A-1] \left[ A \int_t^{t+1} \frac{r(s)}{K(s)} \textnormal{exp}\left( -\int_s^{t+1} r(\tau)d\tau \right)ds\right]^{-1} \\
&& \qquad = [(1-E)A-1] \left[ A  \lim_{t \to t_0+k+0} \int_t^{t+1} \frac{r(s)}{K(s)} \textnormal{exp}\left( -\int_s^{t+1} r(\tau)d\tau \right)ds\right]^{-1}.
\end{eqnarray*}
Here
\begin{eqnarray*}
&& \lim_{t \to t_0+k+0} \int_t^{t+1} \frac{r(s)}{K(s)} \textnormal{exp}\left( -\int_s^{t+1} r(\tau)d\tau \right)ds \\
&& \qquad = \lim_{t \to t_0+k+0} \left[ \int_t^{t_0+k+1} \frac{r(s)}{K(s)} \textnormal{exp}\left( -\int_s^{t+1} r(\tau)d\tau \right)ds \right. \\
&& \qquad \qquad \left. + \int_{t_0+k+1}^{t+1} \frac{r(s)}{K(s)} \textnormal{exp}\left( -\int_s^{t+1} r(\tau)d\tau \right)ds \right] \\
&& \qquad = \lim_{t \to t_0+k+0} \int_t^{t_0+k+1} \frac{r(s)}{K(s)} \textnormal{exp}\left( -\int_s^{t+1} r(\tau)d\tau \right)ds \\
&& \qquad \qquad +  \lim_{t \to t_0+k+0} \int_{t_0+k+1}^{t+1} \frac{r(s)}{K(s)} \textnormal{exp}\left( -\int_s^{t+1} r(\tau)d\tau \right)ds \\
&& \qquad = \lim_{t \to t_0+k+0} \int_t^{t_0+k+1} \frac{r(s)}{K(s)} \textnormal{exp}\left( -\int_s^{t_0+k+1} r(\tau)d\tau - \int_{t_0+k+1}^{t+1} r(\tau)d\tau\right)ds \\
&& \qquad \qquad + \int_{t_0+k+1}^{t_0+k+1} \frac{r(s)}{K(s)} \textnormal{exp}\left( -\int_s^{t_0+k+1} r(\tau)d\tau \right)ds \\
&& \qquad = \lim_{t \to t_0+k+0} \textnormal{exp}\left( -\int_{t_0+k+1}^{t+1} r(\tau)d\tau\right) \\
&& \qquad \qquad \cdot \int_t^{t_0+k+1} \frac{r(s)}{K(s)} \textnormal{exp}\left( -\int_s^{t_0+k+1} r(\tau)d\tau \right)ds \\
&& \qquad = \textnormal{exp}\left( -\int_{t_0+k+1}^{t_0+k+1} r(\tau)d\tau\right) \cdot \int_{t_0+k}^{t_0+k+1} \frac{r(s)}{K(s)} \textnormal{exp}\left( -\int_s^{t_0+k+1} r(\tau)d\tau \right)ds = B
\end{eqnarray*} 
Thus we have  
\begin{equation*}
x_*\left((t_0+k)^+\right) = \frac{(1-E)A-1}{AB}.
\end{equation*}  
On the other hand
\begin{eqnarray*}
&&x_*\left((t_0+k)^-\right) = \lim_{t \to t_0+k-0} x_*(t) \\
&& \qquad = \lim_{t \to t_0+k-0} [(1-E)A-1] \left[ A \int_t^{t+1} \frac{r(s)}{K(s)} \textnormal{exp}\left( -\int_s^{t+1} r(\tau)d\tau \right)ds\right]^{-1} \\
&& \qquad = [(1-E)A-1] \left[ A  \lim_{t \to t_0+k-0} \int_t^{t+1} \frac{r(s)}{K(s)} \textnormal{exp}\left( -\int_s^{t+1} r(\tau)d\tau \right)ds\right]^{-1}.
\end{eqnarray*}     			            
Here
\begin{eqnarray*}
&& \lim_{t \to t_0+k-0} \int_t^{t+1} \frac{r(s)}{K(s)} \textnormal{exp}\left( -\int_s^{t+1} r(\tau)d\tau \right)ds \\
&& \qquad = \lim_{t \to t_0+k-0} \left[ \int_t^{t_0+k} \frac{r(s)}{K(s)} \textnormal{exp}\left( -\int_s^{t+1} r(\tau)d\tau \right)ds \right. \\
&& \qquad \qquad \left. + \int_{t_0+k}^{t+1} \frac{r(s)}{K(s)} \textnormal{exp}\left( -\int_s^{t+1} r(\tau)d\tau \right)ds \right] \\
&& \qquad = \lim_{t \to t_0+k-0} \int_t^{t_0+k} \frac{r(s)}{K(s)} \textnormal{exp}\left( -\int_s^{t+1} r(\tau)d\tau \right)ds \\
&& \qquad \qquad +  \lim_{t \to t_0+k-0} \int_{t_0+k}^{t+1} \frac{r(s)}{K(s)} \textnormal{exp}\left( -\int_s^{t+1} r(\tau)d\tau \right)ds \\
&& \qquad = \lim_{t \to t_0+k+0} \int_t^{t_0+k} \frac{r(s)}{K(s)} \textnormal{exp}\left( -\int_s^{t_0+k} r(\tau)d\tau - \int_{t_0+k}^{t+1} r(\tau)d\tau\right)ds \\
&& \qquad \qquad + \int_{t_0+k}^{t_0+k+1} \frac{r(s)}{K(s)} \textnormal{exp}\left( -\int_s^{t_0+k+1} r(\tau)d\tau \right)ds \\
&& \qquad = \lim_{t \to t_0+k+0} \textnormal{exp}\left( -\int_{t_0+k}^{t+1} r(\tau)d\tau\right) \\
&& \qquad \qquad \cdot \int_t^{t_0+k} \frac{r(s)}{K(s)} \textnormal{exp}\left( -\int_s^{t_0+k} r(\tau)d\tau \right)ds+B \\
&& \qquad = \textnormal{exp}\left( -\int_{t_0+k}^{t_0+k+1} r(\tau)d\tau\right) \cdot \int_{t_0+k}^{t_0+k} \frac{r(s)}{K(s)} \textnormal{exp}\left( -\int_s^{t_0+k} r(\tau)d\tau \right)ds + B \\
&& \qquad = B.
\end{eqnarray*}  
Therefore
\begin{equation*}
x_*\left((t_0+k)^-\right) = \frac{(1-E)A-1}{AB}.
\end{equation*}  
Thus 
\begin{equation*}
x_*\left((t_0+k)^+\right) = x_*\left((t_0+k)^-\right).\\
\end{equation*}  

In this paper is provided a new representation of periodic solution to the impulsive Logistic equation (a) and (b).

%
%
%
%

\section{Main Results}
We consider the impulsive equation \eqref{eq1} where $K(t), r(t) \in PC(R, R)$ are positive and 
periodic functions with the period 1 and
\begin{eqnarray*}
PC(R, R) & = & \{ \sigma: R \to R : \sigma(t) ~ \textnormal{is a piecewise continuous function which} \\
&& ~~ \textnormal{has discontinuities} ~ t_k ~ \textnormal{where} ~ \sigma(t_k-0) ~ \textnormal{and} ~ \sigma(t_k+0) ~ \textnormal{exist}\\
&&  ~~ \textnormal{and} ~\sigma(t_k-0) = \sigma(t_k) ~ \textnormal{holds} \}.
\end{eqnarray*} 
We assume that
\begin{eqnarray*}
&& 0<t_0<t_0+1<t_0+2< \cdots < t_0+k<\cdots, \\
&& \lim_{k \to \infty} \tau_k = +\infty, ~ k=1, 2, \cdots
\end{eqnarray*} 
and $1-E>0$ as in \cite{zha1}.

\begin{thm} \label{thr1}
\textnormal{\cite{zha1}} The solution $x(t)$ with $x(t_0) = x_0$ to \eqref{eq1} is provided in $t \in [t_0+k, t_0+k+1)$ as follows:
\begin{eqnarray}
x(t) & = & \left[ \frac{1}{x_0[(1-E)A]^k} \textnormal{exp} \left(-\int_{t_0+k}^t r(\tau)d\tau \right) + AB \sum_{j=1}^k \frac{1}{[(1-E)A]^j} \right. \nonumber \\
&&  \left. \vphantom{\sum_{j=1}^k} \times \textnormal{exp} \left(-\int_{t_0+k}^t r(\tau)d\tau \right) + \int_{t_0+k}^t \frac{r(s)}{K(s)} \textnormal{exp} \left(-\int_{s}^t r(\tau)d\tau \right)ds \right]^{-1}, \quad \label{eq3} \\
&&  ~~ k=0, 1, 2, \cdots. \nonumber 
\end{eqnarray} 
Here $A$ and $B$ are given as in \eqref{eq2}. 
\end{thm}
The proof of this theorem is the same as in \cite{zha1}.

%
%

\begin{thm} \label{thr2}
If $(1-E)A>1$, then \eqref{eq1} has the unique positive periodic solution with the period 1 provided as follows:
\begin{eqnarray}
x_*(t) & = & [(1-E)A-1] \left[ AB \textnormal{exp} \left(-\int_{t_0+k}^t r(\tau)d\tau \right) \right. \nonumber \\
&& \left. +  [(1-E)A-1] \int_{t_0+k}^t \frac{r(s)}{K(s)} \textnormal{exp} \left(-\int_{s}^t r(\tau)d\tau \right)ds \right]^{-1}  \label{eq4} \\
&&  ~~ t \in [t_0+k, t_0+k+1), ~ k=0, 1, 2, \cdots. \nonumber 
\end{eqnarray} 
\end{thm}
\begin{proof}
From \eqref{eq3}, the solution $x(t)$ to \eqref{eq1} for $t \in [t_0, t_0+1)$ is provided as 
\begin{equation*}
x(t) = \left[ \frac{1}{x_0} \textnormal{exp} \left(-\int_{t_0}^t r(\tau)d\tau \right) + \int_{t_0}^t \frac{r(s)}{K(s)} \textnormal{exp} \left(-\int_{s}^t r(\tau)d\tau \right)ds \right]^{-1}.
\end{equation*}
Therefore we have
\begin{eqnarray*}
x(t+1) & = & \left[ \frac{1}{x_0(1-E)A} \textnormal{exp} \left(-\int_{t_0+1}^{t+t_0+1} r(\tau)d\tau \right) \right. \\
&& + \frac{B}{1-E} \textnormal{exp} \left(-\int_{t_0+1}^{t+t_0+1} r(\tau)d\tau \right) \\
&& + \left. \int_{t_0+1}^{t+t_0+1} \frac{r(s)}{K(s)} \textnormal{exp} \left(-\int_{s}^{t+t_0+1} r(\tau)d\tau \right)ds \right]^{-1}.
\end{eqnarray*}
We find the $x_0$ such that $x(t) = x(t+1)$:
\begin{equation*}
x_0 = \frac{(1-E)A-1}{AB}.
\end{equation*}
Substitute $x_0$ into \eqref{eq3}, we have
\begin{eqnarray*}
x_*(t) & = & \left[ \frac{AB}{[(1-E)A-1][(1-E)A]^k} \textnormal{exp} \left(-\int_{t_0+k}^{t} r(\tau)d\tau \right) \right. \\
&& \qquad + AB\sum_{j=1}^k \frac{1}{[(1-E)A]^j} \textnormal{exp} \left(-\int_{t_0+k}^{t} r(\tau)d\tau \right)  \\
&& \qquad + \left. \int_{t_0+k}^{t} \frac{r(s)}{K(s)} \textnormal{exp} \left(-\int_{s}^{t} r(\tau)d\tau \right)ds \right]^{-1} \\
& = & \left[ \frac{1+\sum_{j=0}^{k-1} [(1-E)A]^j [(1-E)A-1]}{[(1-E)A-1][(1-E)A]^k} AB \textnormal{exp} \left(-\int_{t_0+k}^{t} r(\tau)d\tau \right) \right. \\
&& \qquad + \left. \int_{t_0+k}^{t} \frac{r(s)}{K(s)} \textnormal{exp} \left(-\int_{s}^{t} r(\tau)d\tau \right)ds \right]^{-1} \\
& = & [(1-E)A-1] \left[ AB \textnormal{exp} \left(-\int_{t_0+k}^{t} r(\tau)d\tau \right) \right. \\
&& \qquad + \left.  [(1-E)A-1] \int_{t_0+k}^{t} \frac{r(s)}{K(s)} \textnormal{exp} \left(-\int_{s}^{t} r(\tau)d\tau \right)ds \right]^{-1}. 
\end{eqnarray*}
Thus we have \eqref{eq4}. 

Now we show that \eqref{eq4} is a periodic solution to \eqref{eq1}. 
It is evident that \eqref{eq4} satisfies (a) and (b) from Theorem \ref{thr1} but we will show that 
\eqref{eq4} satisfies the impulse condition (b) in detail.  

When $t \in [t_0, t_0+1)$, we have
\begin{eqnarray*}
x_*(t) & = & [(1-E)A-1] \left[AB \textnormal{exp} \left(-\int_{t_0}^{t} r(\tau)d\tau \right) \right. \\
&& +  \left. [(1-E)A-1] \int_{t_0}^{t} \frac{r(s)}{K(s)} \textnormal{exp} \left(-\int_{s}^{t} r(\tau)d\tau \right)ds \right]^{-1}.
\end{eqnarray*}
Thus we have
\begin{eqnarray*}
x\left( (t_0+1)^- \right) & = & \lim_{t \to t_0+1-0} x_*(t) \\
& = & [(1-E)A-1] \left[AB \lim_{t \to t_0+1-0} \textnormal{exp} \left(-\int_{t_0}^{t} r(\tau)d\tau \right) \right. \\
&& +  \left. [(1-E)A-1] \lim_{t \to t_0+1-0} \int_{t_0}^{t} \frac{r(s)}{K(s)} \textnormal{exp} \left(-\int_{s}^{t} r(\tau)d\tau \right)ds \right]^{-1}.
\end{eqnarray*}
Here
\begin{equation*}
 \lim_{t \to t_0+1-0} \textnormal{exp} \left(-\int_{t_0}^{t} r(\tau)d\tau \right) = \textnormal{exp} \left(-\int_{t_0}^{t_0+1} r(\tau)d\tau \right) = A^{-1}
\end{equation*}
and
\begin{eqnarray*}
&& \lim_{t \to t_0+1-0} \int_{t_0}^{t} \frac{r(s)}{K(s)} \textnormal{exp} \left(-\int_{s}^{t} r(\tau)d\tau \right)ds \\
&& \qquad \qquad = \int_{t_0}^{t_0+1} \frac{r(s)}{K(s)} \textnormal{exp} \left(-\int_{s}^{t_0+1} r(\tau)d\tau \right)ds = B.
\end{eqnarray*}
Therefore we have
\begin{eqnarray*}
x\left( (t_0+1)^- \right) & = & [(1-E)A-1] \left\{ ABA^{-1} + [(1-E)A-1]B \right\}^{-1} \\
& = & \frac{(1-E)A-1}{AB(1-E)}.
\end{eqnarray*}
When $t \in [t_0+1, t_0+2)$, we have
\begin{eqnarray*}
x_*(t) & = & [(1-E)A-1] \left[AB \textnormal{exp} \left(-\int_{t_0+1}^{t} r(\tau)d\tau \right) \right. \\
&& +  \left. [(1-E)A-1] \int_{t_0+1}^{t} \frac{r(s)}{K(s)} \textnormal{exp} \left(-\int_{s}^{t} r(\tau)d\tau \right)ds \right]^{-1}.
\end{eqnarray*}
Calculate $x\left( (t_0+1)^+ \right) = \lim_{t \to t_0+1+0} x_*(t)$. If we consider 
\begin{equation*}
 \lim_{t \to t_0+1+0} \textnormal{exp} \left(-\int_{t_0+1}^{t} r(\tau)d\tau \right) = \textnormal{exp} \left(-\int_{t_0+1}^{t_0+1} r(\tau)d\tau \right) = 1
\end{equation*}
and
\begin{eqnarray*}
&& \lim_{t \to t_0+1+0} \int_{t_0+1}^{t} \frac{r(s)}{K(s)} \textnormal{exp} \left(-\int_{s}^{t} r(\tau)d\tau \right)ds \\
&& \qquad \qquad = \int_{t_0+1}^{t_0+1} \frac{r(s)}{K(s)} \textnormal{exp} \left(-\int_{s}^{t_0+1} r(\tau)d\tau \right)ds = 0,
\end{eqnarray*}
then we have
\begin{eqnarray*}
x_*\left( (t_0+1)^+ \right) & = & [(1-E)A-1] \left\{ AB + [(1-E)A-1] \cdot 0 \right\}^{-1} \\
& = & \frac{(1-E)A-1}{AB}.
\end{eqnarray*}
Therefor we have
\begin{eqnarray*}
&& x_*\left( (t_0+1)^+ \right) -  x_*\left( (t_0+1)^- \right) = \frac{(1-E)A-1}{AB} - \frac{(1-E)A-1}{AB(1-E)} \\
&& \qquad \qquad = \frac{(1-E)A-1}{AB} \left( 1 - \frac{1}{1-E} \right) = -E \frac{(1-E)A-1}{AB(1-E)} \\
&& \qquad \qquad = -E \cdot x_*\left( (t_0+1)^- \right).
\end{eqnarray*}

Next we show that \eqref{eq4} has periodicity. When $t \in [t_0+k, t_0+k+1)$, we have $t+1 \in [t_0+k+1, t_0+k+2)$,
\begin{eqnarray*}
&& x_*(t) = [(1-E)A-1] \left[AB \textnormal{exp} \left(-\int_{t_0+k}^{t} r(\tau)d\tau \right) \right. \\
&& \qquad \qquad +  \left. [(1-E)A-1] \int_{t_0+k}^{t} \frac{r(s)}{K(s)} \textnormal{exp} \left(-\int_{s}^{t} r(\tau)d\tau \right)ds \right]^{-1}, \\
&& x_*(t+1) = [(1-E)A-1] \left[AB \textnormal{exp} \left(-\int_{t_0+k+1}^{t+1} r(\tau)d\tau \right) \right. \\
&& \qquad \qquad +  \left. [(1-E)A-1] \int_{t_0+k+1}^{t+1} \frac{r(s)}{K(s)} \textnormal{exp} \left(-\int_{s}^{t+1} r(\tau)d\tau \right)ds \right]^{-1}.
\end{eqnarray*}
From the periodicity of $K(s), r(s)$, we have
\begin{equation*}
 \textnormal{exp} \left(-\int_{t_0+k}^{t} r(\tau)d\tau \right) = \textnormal{exp} \left(-\int_{t_0+k+1}^{t+1} r(\tau)d\tau \right)
\end{equation*}
and thus we have
\begin{equation*}
\int_{t_0+k}^{t} \frac{r(s)}{K(s)} \textnormal{exp} \left(-\int_{s}^{t} r(\tau)d\tau \right)ds = \int_{t_0+k+1}^{t+1} \frac{r(s)}{K(s)} \textnormal{exp} \left(-\int_{s}^{t+1} r(\tau)d\tau \right)ds.
\end{equation*}
Therefore we have $x_*(t) = x_*(t+1)$.
\end{proof}

 \end{document}